\documentclass[11pt, twocolumn]{article}
\usepackage{article}
\author{Torgeir Aamb\o \thanks{Norwegian Defence Research Establishment (FFI)}}
\title{Possibilistic operators in Formal Concept Analysis as Kan extensions}
\date{}

\begin{document}
\maketitle

\begin{abstract}
    \textbf{In this paper we prove that Dubois--Prade's eight possibilistic operators in Formal Concept Analysis arise canonically from Kan extensions of the underlying boolean profunctor. This provides a conceptual explanation for the result that $N\Pi$-pairs are the formal concepts of the complement context. We further prove that the FCA closure operator and the $N\Pi$-pairs are the only symmetric or asymmetric operator compositions that give formal concepts. Finally we use these eight possibilistic operators to construct new closure operators on a formal context via standard categorical arguments.} 
\end{abstract}

\section{Introduction}

Formal concept analysis (FCA) is a technique for structuring data by associated attributes \cite{ganter-wille_1999}, and has during recent years seen a wide array of applications in various contexts \cite{singh-kumar-gani_2016}. Given a list of objects, a list of attributes, and a relation describing which object has which attribute, FCA constructs a set of \emph{formal concepts} which are naturally structured into a complete lattice. 

Formal concepts are constructed as fixedpoints of derivation operations $(-)'$ on the set of objects and attributes. More formally, given a set of objects $G$, a set of attributes $M$ and a relation $R\subseteq G\times M$, one defines for a subset of objects $A\subseteq G$ its \emph{intent}
\[A' = \{m \in M \mid \forall g\in G, (g\in A \implies gRm)\},\]
and similarly for a subset of attributes $B\subseteq M$ its \emph{extent}
\[B' = \{g \in G \mid \forall m \in M, (m\in B \implies gRm)\},\]
where $xRy$ means that the object $x$ has the attribute $y$. In words, the extent of $A$ is the set of common attributes shared by all of its objects, and the intent of $B$ is the set of objects that have all of its attributes. A formal concept is then a subset $(A,B)\subseteq G\times M$ such that $B=A'$ and $A=B'$, see \cite{ganter-wille_1999} for details. 

In \cite{dubois-dupin-prade_2007, dubois-prade_2009}, the authors take inspiration from possibility theory -- see \cite{dubois-prade_1988, dubois-prade_1998, dubois-prade_2015} -- to give a new reading of additional power-set operators on $P(G)$, previously studied in relation to qualitative data analysis \cite{duntsch-gediga_2002} and rough set theory \cite{yao-chen_2006}. 

Denoting the previous FCA operator by $(-)' = (-)^\Delta$, they further study for a subset $A\subseteq G$, the operators 
\[A^\Pi = \{m \in M \mid \exists g\in G, (g\in A \wedge gRm)\},\]
as well as
\[A^\nabla = \{m\in M \mid \exists g\in G, (g\not\in A \wedge g\oR m)\},\]
where $\oR = \neg R$ is the negated relation, and finally
\[A^N = \{m\in M \mid \forall g\in G, (gRm \implies g\in A)\}.\]
These are usually called the potential possibility $\Pi$, the guaranteed necessity $N$, the guaranteed possibility $\Delta$ and the potential necessity $\nabla$. One can similarly define these four operators for subsets $B\subseteq M$, which we denote with subscripts, $(-)_\Pi, (-)_N, (-)_\Delta$ and $(-)_\nabla$. We will for the rest of the paper refer to these as the eight \emph{possibilistic operators}.  

By using a well established categorical framework for FCA via boolean profunctors, the first goal of this paper is to prove that the above eight operators arise canonically from the underlying relation $R$, rather than from possibility-theoretic considerations. More precisely, the relation $R$ and its complement relation $\oR = \neg R$ both canonically define four operations each via right and left Kan extensions of curryings of the boolean profunctor. Together these form four adjoint pairs on $P(G)$ and $P(M)$ -- the powersets of $G$ and $M$. 

Our main result is the following theorem. 

\begin{introthm}
    Any formal context has eight canonical associated functors arising from Kan extensions of the boolean profunctor $R$ and its complement $\oR$. These correspond naturally to the eight possibilistic operators in the following manner:
    \begin{multicols}{2}
    \begin{enumerate}
        \item $(-)^\Delta = \Ran_y(R^-)$
        \item $(-)_\Delta = \Ran_y(R_-)$
        \item $(-)^\Pi = \Lan_y(R^-)$
        \item $(-)_\Pi = \Lan_y(R_-)$
        \item $(-)^N = \Ran_y(\oR^-)\neg$
        \item $(-)_N = \Ran_y(\oR_-)\neg$
        \item $(-)^\nabla = \Lan_y(\oR^-)\neg$
        \item $(-)_\nabla = \Lan_y(\oR_-)\neg$
    \end{enumerate}
    \end{multicols}
    Furthermore, these form adjoint pairs $(-)^\Delta\dashv (-)_\Delta$, $(-)^\Pi\dashv (-)_N$, $(-)_\Pi\dashv (-)^N$ and $(-)_\nabla\dashv (-)^\nabla$. 
\end{introthm}

While the derivation operators have previously been understood categorically, see for example \cite{ferrer_2023}, the remaining six possibilistic operators apparently have not. This categorical viewpoint makes the properties of these operators, as studied in for example \cite{djouadi_prade_2011}, evident from standard categorical arguments. 

By design, fixedpoints of $(-)_\Delta (-)^\Delta : P(G) \to P(G)$ are exactly the extents of the formal concepts of the formal context $(G,M,R)$. In \cite{djouadi_prade_2011, yakub-djouadi-dubois-prade_2016}, the authors prove that the fixedpoints for the asymmetric operator composition $(-)_\Pi (-)^N$, which they call $N\Pi$-pairs, are precisely the formal concepts of the complementary formal context $(G,M,\oR)$ defined using the negated relation $\oR$. Via the categorical viewpoint, this result is given its \emph{raison d'
être}. 

The authors in \cite{yakub-djouadi-dubois-prade_2016} note at the end of the paper that further research should be done to understand the remaining asymmetric operator compositions of the possibilistic operators. We prove that the operators $(-)_\Delta (-)^\Delta$ and $(-)_\Pi (-)^N$ are the only ones that has formal concepts as their fixed points. Along the way we provide some standard categorical tools for studying such compositions, as well as other endofunctors on $P(G)$. In particular, any composition-sequence of the possibilistic left adjoints has a corresponding closure operator, giving a plethora of new types of ``concepts'' one can study.  

\textbf{Main contributions:} The paper provides four new contributions to formal concept analysis. First, it provides a good reason for \emph{why} the eight possibilistic operators are as natural to study as standard FCA closure operators. Second, it provides a coherent framework for several existing facts about possibilistic operators in formal concept analysis, like the fact that $N\Pi$-pairs are the concepts of the complement context. Third, it shows how general machinery can identify new closure operators on formal contexts. Fourth, it sets the stage for further fuzzy, possibilistic, quantale-enriched versions of the same results, all packaged into the same underlying theory of enriched categories.

\section{Category theory and FCA}

Formal concept analysis has attracted interest by category theorists for being a natural and useful application of their abstract tools. Category theory was introduced by Eilenberg and MacLane in \cite{eilenberg-maclane_1945}, and is an abstract mathematical theory for understanding compositional structure -- see \cite{maclane_1998} for a comprehensive introduction. 

A category consists of \emph{objects} and \emph{morphisms}, which are abstract relationships between objects. Category theory can then be thought of as studying objects by understanding how they relate to all other objects of the same defined type. We will assume that the reader is familiar with basic categorical constructions. 

For a category theorist, the objects $G$ and attributes $M$ can be viewed as discrete categories. Any relation can be described as a $\2 = \{0,1\}$-valued profunctor, also called \emph{boolean} profunctors. In our case the profunctor $R\: G\nrightarrow M$ is a functor $G\op\times M\to \2$, which takes a pair $(g,m)$ to $1$ if $g$ has the attribute $m$ -- see \cite{ferrer_2023} for details. This will be our running definition of a formal context. 

\begin{definition}
    A \emph{formal context} $(G,M,R)$ consists of two discrete categories $G$ and $M$ and a boolean profunctor $R\: G\nrightarrow M$. We will say an object $g\in G$ has attribute $m\in M$ if $R(g,m)=1$, usually written $gRm$. 
\end{definition}

From this perspective we can also produce the formal concepts associated to any formal context, as we show in the next section. This is fairly standard in categorical FCA. 

\subsection{The closure operators}

Fix a formal context $(G,M,R)$. To construct the associated formal concepts, notice that the profunctor $R$ is by currying equivalent to two functors 
\[R^{-}\: G \to [M, \2]\op \text{ and } R_{-}\: M \to [G\op, \2].\]
We also have two Yoneda embeddings, giving a diagram
\begin{center}
\begin{tikzcd}
G \arrow[r, "R^{-}"] \arrow[d, "y"'] & {[M,\2]\op} \\
{[G\op, \2]} & M \arrow[l, "R_{-}"] \arrow[u, "cy"']       
\end{tikzcd}
\end{center}
where $y$ is the Yoneda embedding and $cy$ is the Yoneda coembedding into the presheaf and opcopresheaf categories respectively. As $G$ and $M$ are discrete, the functor categories are precisely the powersets of $G$ and $M$ respectively: $[G\op, \2] = P(G)$ and $[M, \2]\op = P(M)\op$. 

\begin{remark}
    As there are a lot of notational embellishments in this paper we here give a brief overview of the notation we will use. The relation, as well as its associated boolean profunctor is always denoted by $R$. The two curryings are denoted $R^-$ and $R_-$. Their right Kan extensions will be denoted by $R^*$ and $R_*$ respectively, while their left Kan extensions are denoted by $R^!$ and $R_!$ respectively. Hence, all functors with a superscripted embellishment denotes an operator on objects, while all functors with a subscript embellishment denotes an operator on attributes. The same notation is used for the complement relation $\oR$. 
\end{remark}

As $[G\op, \2]$ is complete, there is a right Kan extension of $R^-$ along the Yoneda embedding $y$. The following lemma shows that this gives the standard FCA derivation on $P(G)$. 

\begin{lemma}
    \label{lm:upper-Delta}
    The right Kan extension $R^*$ of $R^-$ is the standard FCA derivation operator $(-)^\Delta$. 
\end{lemma}
\begin{proof}
    By \cite[Section X.4]{maclane_1998} (see also \cite[Section 4]{kelly_1982}), the functor $R^*$ is given pointwise by a categorical end-formula, which for a presheaf $A\in [G\op, \2]$ is
    \[R^*(A)(m) = \int_{g\in G} \2(A(g), gRm).\]
    In $\2$, the internal hom is given by \emph{implication}, reducing it to the conjunction 
    \[R^*(A)(m) = \bigwedge_{g\in G}(A(g) \implies gRm).\]
    In other words, having $m\in R^*(A)$ is equivalent to having that for all $g$, if $g\in A$ then $g$ has the attribute $m$. This is precisely the FCA derivation operation $(-)^\Delta$ on $P(G)$ as defined earlier. 
\end{proof}

Similarly, the other curried functor $R_{-}$ has a right Kan extension $R_*$ along the coYoneda embedding $cy\: M\to [M,\2]\op$. This gives the other standard FCA derivation operation, as the following lemma shows.  

\begin{lemma}
    \label{lm:lower-Delta}
    The right Kan extension $R_*$ of $R_-$ is the standard FCA derivation operator $(-)_\Delta$. 
\end{lemma}
\begin{proof}
    The proof is completely analogous to \cref{lm:upper-Delta}, as the right Kan extension has the formula
    \[R_*(B)(g)=\int_{m\in M}\2(B(m), gRm).\]
    Hence, the same argument works, giving $R_* = (-)_\Delta$.
\end{proof}

These two functors are not unrelated, but form an adjoint pair on the powerset categories. 

\begin{lemma}
    \label{lm:adjoint-Delta}
    The two operators $R^*$ and $R_*$ form an adjoint pair
    \begin{center}
    \begin{tikzcd}
    {[G\op, \2]} \arrow[r, "R^*", yshift=2pt] & {[M, \2]\op} \arrow[l, "R_*", yshift=-2pt].
    \end{tikzcd}
    \end{center}
\end{lemma}
\begin{proof}
    We need to show that for any objects $A\in [G\op, \2]$ and $B\in [M, \2]\op$, we have $R^*(A)\leq B$ if and only if $A\leq R_*(B)$. By the definition of $R^*$, we have $R^*(A)\leq B$ if and only if for all $m\in M$, 
    \[\bigwedge_{g\in G}(A(g)\implies gRm)\leq B(m).\]
    The meet is below if and only if each factor is, hence this is equivalent to the statement that for all $m\in M$ and $g\in G$,
    \[(A(g)\implies gRm) \leq B(m).\]
    By residuation, this is again equivalent to the statement that for all $g\in G$, 
    \[A(g) \leq \bigwedge_{m\in M} (B(m) \implies gRm).\]
    The right hand side is the pointwise definition of $R_*$, giving finally $A\leq R_*(B)$. 
\end{proof}

We can now, as is standard in categorical FCA, recognize the formal concepts of $(G,M,R)$ as the fixedpoints of the adjunction $R^*\dashv R_*$, often called the \emph{nucleus} of the profunctor $R$, see \cite[Chapter 5]{ferrer_2023}. 

More formally, the composition $R_* R^*$ is a monad on $P(G)$, and the extents of the formal concepts are the objects in the Eilenberg--Moore category of this monad. Similarly, the composition $R^* R_*$ is a comonad on $[M,\2]\op$, and hence a monad on $[M,\2] = P(M)$. Its Eilenberg--Moore category gives the intents of the formal concepts. The nucleus consists of pairs of such objects and attributes, in other words: an object in $\mathrm{Nuc}(R)$ is of the form $(A,B)$ where $A = R_*(R^*(B))$ and $B=R^* (R_* (A))$. By adjunction, however, these determine each other, so we can focus only on one of them, choosing throughout the paper to focus on the object side. 

To summarize, the formal concepts arise completely naturally via standard categorical constructions from the profunctor $R$ via the fixedpoints of the monad $R_*R^*$ on $[G\op, \2]$, defined using right Kan extensions of the two associated curryings. 

\subsection{Two more operators}

Above we defined the functors $R^*$ and $R_*$ as the right Kan extensions of the two curried functors $R^-$ and $R_-$ along the respective Yoneda embeddings. As $[G\op, \2]$ also is cocomplete, this prompts the idea to further consider the \emph{left} Kan extensions of the two curried functors. This will give us another two of the eight operations described by Dubois--Prade. 

\begin{lemma}
    \label{lm:upper-Pi}
    Given a formal context $R\: G\nrightarrow M$, the left Kan extension $R^!$ of the functor $R^-$ coincides with the possibilistic operator $(-)^\Pi$.  
\end{lemma}
\begin{proof}
    For the functor $R^-$, the left Kan extension is defined by the coend formula
    \[R^!(A)(m) = \int^{g\in G}(A(g)\otimes gRm).\]
    The monoidal structure in $\2$ is the \emph{and} operation, and as $[M,\2]\op$ is a complete lattice, the coend is the disjunction
    \[R^!(A)(m) = \bigvee_{g\in G}(A(G)\wedge gRm).\]
    In words, we have $m\in R^!(A)$ if and only if there is some $g\in A$ such that $g$ has the attribute $m$, which is precisely the definition of the operator $(-)^\Pi$.
\end{proof}

Similar to before, we also have the left Kan extension of the other curried functor $R_-$. As above, this coincides with the possibilistic operator $(-)_\Pi$. 

\begin{lemma}
    \label{lm:lower-Pi}
    The left Kan extension $R_!$ of the functor $R_-$ coincides with the possibilistic operator $(-)_\Pi$. 
\end{lemma}
\begin{proof}
    The left Kan extension is given by
    \[R_!(B)(g) = \int^{m\in M}(B(m)\otimes gRm),\]
    which is equivalent to 
    \[R_!(B)(g)=\bigvee_{m\in M}B(m)\wedge gRm,\]
    which is the pointwise definition of $(-)_\Pi$. 
\end{proof}

As opposed to the previous two functors $(-)^\Delta$ and $(-)^\Delta$, these are not adjoint to each other. We will later, however, prove that they are part of adjoint pairs.

\subsection{The complement profunctor}

For any boolean profunctor $R$ one can also define its \emph{complement} $\oR\: G\nrightarrow M$, defined by $g\oR m := \neg gRm$. This defines the \emph{complement} formal context $(G,M,\oR)$ as discussed in \cite{djouadi_prade_2011,yakub-djouadi-dubois-prade_2016}. 

As $\oR$ is a boolean profunctor in its own right, it also has four associated functors, $\oR^*$, $\oR_*$, $\oR^!$ and $\oR_!$, coming from the right and left Kan extensions. In the following lemmas we show that these give the remaining four possibilistic operators via a negation duality $\neg$ between operators on $(G,M,R)$ and operators on the complement context $(G,M,\oR)$. 

The negation is defined pointwise on any presheaf $A$ as $\neg A(g)$. Given an operator $F$ on $P(G)$, we will call precomposing with the complement operator the \emph{right dual} of $F$, and denote it $F\neg$. Similarly we define the \emph{left dual} to be $\neg F$. 

\begin{lemma}
    \label{lm:upper-N}
    The right dual of the operator $\oR^*$ coincides with the possibilistic operator $(-)^N$. 
\end{lemma}
\begin{proof}
    The functor is given pointwise by 
    \[\oR^*(\neg A)(m) = \bigwedge_{g\in G}(\neg A(g)\implies g\oR m).\] 
    By definition we have $g\oR m = \neg gRm$, which by the equivalence $\neg P\implies \neg Q \iff Q\implies P$, gives 
    \[\oR^*(\neg A)(m) = \bigwedge_{g\in G}(gRm \implies A(g)),\]
    which is the definition of $(-)^N$. 
    \end{proof}

Similarly, for the Kan extension of the other curried functor $\oR_-$, we obtain the operator $(-)_N$.

\begin{lemma}
    \label{lm:lower-N}
    The right dual $\oR_*\neg$ of the right Kan extension of the curried functor $\oR_-$, coincides with the possibilistic operator $(-)_N$.
\end{lemma}
\begin{proof}
    The proof is the same as \cref{lm:upper-N}. 
\end{proof}

Now we have but two operators left, which arise from the two left Kan extensions associated to the two curryings of $\oR$. The proofs are essentially just using the definitions of the operators. 

\begin{lemma}
    \label{lm:upper-nabla}
    The right dual $\oR^!\neg$ of the left Kan extension of the curried functor $\oR^-$, is equivalent to the possibilistic operator $(-)^\nabla$.
\end{lemma}
\begin{proof}
    By definition we have 
    \[\oR^!(\neg A)(m)=\bigvee_{g\in G} (\neg A(g)\wedge g\oR m).\]
    which is the pointwise definition of $(-)^\nabla$.   
\end{proof}

\begin{lemma}
    \label{lm:lower-nabla}
    The right dual $\oR_!\neg$ of the left Kan extension of the curried functor $\oR_-$, is equivalent to the possibilistic operator $(-)_\nabla$.
\end{lemma}
\begin{proof}
    The proof is the same as \cref{lm:upper-nabla}. 
\end{proof}

To summarize then, we have four pairs of adjoint functors, arising from two boolean profunctors. This gives \cref{tab:lan-ran}, where we have denoted operators as either a right Kan extension (Ran) or a left Kan extension (Lan) and omitted the duals. 

\begin{table}[ht]
\centering
\begin{tabular}{l|ll}
& Ran & Lan\\
\hline
$R^-$   & $(-)^\Delta$ & $(-)^\Pi$ \\
$R_-$   & $(-)_\Delta$ & $(-)_\Pi$  \\
$\oR^-$ & $(-)^N$      &  $(-)^\nabla$ \\
$\oR_-$ & $(-)_N$    & $(-)_\nabla$
\end{tabular}
\caption{Possibilistic operators as Kan extensions}
\label{tab:lan-ran}
\end{table}

\subsection{Properties}

In \cite{djouadi_prade_2011}, several properties and relations between the eight possibilistic operators were proven. In this section we reprove some of these from standard categorical properties. First, however, we prove that this categorical perspective also recovers the De Morgan dualities between these eight operators, as is their usual classical definitions in possibility theory -- the necessity is usually defined as $N(A) = 1-\Pi(\neg A)$, and similarly, $\nabla(A)=1-\Delta(\neg A)$. 

\begin{lemma}
    \label{lm:De-Morgan-duality}
    The eight possibilistic operators are related by De Morgan dualities: $(-)^N = \neg (-)^\Pi \neg$ and $(-)^\nabla = \neg (-)^\Delta \neg$, and similarly for the lower-case operators. 
\end{lemma}
\begin{proof}
    We only prove the relations for the upper-case operators, as the others follow by symmetry. 

    Let $A\in [G\op, \2]$ be a presheaf. By definition we have 
    \[\oR^*(A)(m) = \bigwedge_{g\in G}(A(g)\implies g\oR m).\]
    Using $g \oR m = \neg gRm$ and the boolean equivalence $P\implies \neg Q \iff \neg (P\wedge Q)$, we have 
    \[\oR^*(A)(m) = \bigwedge_{g\in G}\neg (A(g)\wedge g R m).\]
    This is, by De Morgan duality, equivalent to 
    \[\oR^*(A)(m) = \neg\bigvee_{g\in G}(A(g)\wedge g R m),\]
    which we can recognize as $\neg R^!$. 

    By \cref{lm:upper-N} and \cref{lm:upper-Pi} we have $(-)^N = \oR^*\neg$ and $(-)^\Pi = R^!$, giving 
    \[(-)^N = \oR^*\neg = \neg R^!\neg = \neg (-)^\Pi\neg,\]
    proving the first duality. 

    For the second claim we have by definition that 
    \[R^*(A)(m) = \bigwedge_{g\in G}(A(g)\implies gRm),\]
    which by the equivalence $P\implies Q \iff \neg P\vee Q$ gives 
    \[R^*(A)(m) = \bigwedge_{g\in G}(\neg A(g)\vee gRm).\]
    By De Morgan duality, we have a boolean equality $\neg A(g)\vee gRm = \neg(A(g)\wedge g\oR m)$, which by another use of De Morgan duality gives 
    \[R^*(A)(m) = \neg\bigvee_{g\in G}(A(g)\wedge g\oR m),\]
    which by definition is $\neg \oR^!(A)(m).$ Hence, $R^* = \neg \oR^!$. To conclude then, we have
    \[(-)^\Delta = R^* =\neg\oR^! = \neg \oR^! \neg\neg = \neg (-)^\nabla \neg\] 
    which finishes the proof. 
\end{proof}

\begin{remark}
    Purely in terms of the Kan extensions, and not the possibilistic operators, these dualities are: 
    \begin{multicols}{2}
    \begin{enumerate}
        \item $R^* = \neg \oR^!$
        \item $R_* = \neg \oR_!$
        \item $R^! = \neg \oR^*$
        \item $R_! = \neg \oR_*$
    \end{enumerate}
    \end{multicols}
\end{remark}

The following reproves \cite[5.1]{dubois-dupin-prade_2007}. 

\begin{proposition}
    Let $(G,M,R)$ be a formal context. For all subsets $A_1, A_2$ of $G$, we have: 
    \begin{enumerate}
        \item $(A_1\cup A_2)^\Delta = A_1^\Delta \cap A_2^\Delta$,
        \item $(A_1\cap A_2)^N = A_1^N \cap A_2^N$,
        \item $(A_1\cup A_2)^\Pi = A_1^\Pi \cup A_2^\Pi$, and
        \item $(A_1\cap A_2)^\nabla = A_1^\nabla \cup A_2^\nabla$.
    \end{enumerate}
\end{proposition}
\begin{proof}
    We prove these one by one in the order presented. 
    \begin{enumerate}
        \item As $(-)^\Delta\: [G\op, \2]\to [M,\2]\op$ is a left adjoint it preserves colimits, which in $[G\op,\2]$ are unions and in $[M,\2]\op$ are intersections. 
        \item The functor $(-)^N = \oR^*\neg$ is the dual of a left adjoint. As $\neg\: [G\op, \2]\to [G\op, \2]\op$ is an involution, it preserves both limits and colimits. Hence, intersections, which are limits in $[G\op, \2]$, are sent to colimits in $[M,\2]\op$, which are also intersections.
        \item Follows from (2) and the De Morgan duality of \cref{lm:De-Morgan-duality}.
        \item Follows from (1) and \cref{lm:De-Morgan-duality}. \qedhere
    \end{enumerate}
\end{proof}

We have already seen that the two functors $R^*$ and $R_*$ formed an adjoint pair, giving a categorical definition of formal concepts as the Eilenberg--Moore category of the associated idempotent monad. Let us also investigate the other operators. 

\begin{lemma}
    \label{lm:adjoint-Pi-N}
    The two possibilistic operators $(-)^\Pi$ and $(-)_N$ form an adjoint pair
    \begin{center}
    \begin{tikzcd}
    {[G\op, \2]} \arrow[r, "(-)^\Pi", yshift=2pt] & {[M, \2]\op} \arrow[l, "(-)_N", yshift=-2pt].
    \end{tikzcd}
    \end{center}
\end{lemma}
\begin{proof}
    We need to show that for all $A\in [G\op,\2]$ and $B\in [M,\2]\op$ we have $A\leq B_N$ if and only if $A^\Pi\leq B$. 
    
    By definition, $A^\Pi\leq B$ holds if and only if for all attributes $m\in M$ we have $\bigvee_g(A(g)\wedge gRm)\implies B(m)$ -- this is just the pointwise condition. 

    Since implication out of a join is equivalent to
    implication out of each summand, this holds if and only if
    \[\forall g\in G,\;\forall m\in M,\;(A(g)\wedge gRm)\implies B(m)\]

    Rewriting the implication, we get
    \[\forall g\in G,\;A(g)\implies (gRm\implies B(m))\]
    for all $m\in M$, and therefore
    \[\forall g\in G,\;A(g)\implies\bigwedge_{m\in M}(gRm\implies B(m)),\]
    which is the pointwise definition of $(-)_N$. This gives $A^\Pi\leq B$ if and only if $A\leq B_N$, proving that we have an adjunction $(-)^\Pi\dashv (-)_N$.
\end{proof}

From this we automatically obtain one of the main results from \cite{yakub-djouadi-dubois-prade_2016}. The categorical viewpoint reveals, in our opinion, more clearly why this results appears. This was one of our motivations for this paper. 

\begin{corollary}
    An object $A \in [G\op, \2]$ is a fixed point of $(-)_\Pi (-)^N$, i.e. it is a part of an $N\Pi$-pair, if and only if $\neg A$ is a formal concept of $(G,M,\oR)$.
\end{corollary}
\begin{proof}
    A formal concept in $(G,M,\oR)$ is a fixedpoint of $\oR_*\oR^*$. By inserting negations, this is equivalent to having $\neg A = \neg\oR_*\oR^* \neg (\neg A)$. Moreover, we have $\neg\oR_*\oR^*\neg = \neg\oR_* \neg\neg\oR^*\neg$, which by the De Morgan duality of \cref{lm:De-Morgan-duality} gives 
    \[\neg\oR_* \neg \neg \oR^*\neg = \neg (-)_N \neg (-)^N = (-)_\Pi (-)^N,\]
    showing that the $N\Pi$-pairs are exactly the formal concepts on the complement context. 
\end{proof}

\begin{remark}
    Completely analogously, we have that $A$ is part of a $\Pi N$-pair if and only if $A$ is a formal concept of $(G,M,\oR)$. This is also obtained by De Morgan duality: 
    \[\oR_* \oR^*= \oR_*\neg\neg \oR^* \neg\neg= (-)_N \neg (-)^N\neg = (-)_N (-)^\Pi.\]
\end{remark}

The other two functors $(-)^N$ and $(-)_\Pi$, as well as the functors $(-)^\nabla$ and $(-)_\nabla$, also form adjoint pairs. This follows from the De Morgan dualities of \cref{lm:De-Morgan-duality} together with the following lemma. 

\begin{lemma}
    Let $\C$ and $\D$ be lattices with an order-reversing involution $\neg$. If $F\dashv G\: \D \rightleftarrows \C$ is an adjoint pair of functors, then their De Morgan duals form an adjoint pair $\neg G\neg \dashv \neg F\neg$. 
\end{lemma}
\begin{proof}
    Assume $\neg G(\neg B)\leq A$. As $\neg$ is order-reversing, this is equivalent to $\neg A \leq  G(\neg B)$. As $F$ and $G$ are adjoint, this is equivalent to $F(\neg A) \leq \neg B$, which by the order-reversing negation gives finally $B \leq \neg F(\neg A)$.
\end{proof}

This immediately gives the following two corollaries from \cref{lm:De-Morgan-duality} together with \cref{lm:adjoint-N-Pi} and \cref{lm:adjoint-Delta} respectively. 

\begin{corollary}
    \label{lm:adjoint-N-Pi}
    The two possibilistic operators $(-)^N$ and $(-)_\Pi$ form an adjoint pair $(-)_\Pi \dashv (-)^N$.
\end{corollary} 

\begin{corollary}
    \label{lm:adjoint-nabla}
    The two possibilistic operators $(-)^\nabla$ and $(-)_\nabla$ form an adjoint pair $(-)_\nabla \dashv (-)^\nabla$.
\end{corollary}

We then get the usual inclusion properties of the possibilistic operators from standard categorical observations. 

\begin{proposition}
    \label{prop:closure-interior}
    Let $(G,M,R)$ be a formal context. For all subsets $A$ of $G$, we have: 
    \begin{multicols}{2}
    \begin{enumerate}
        \item $A\leq A^\Delta_\Delta$
        \item $A\leq A^\Pi_N$
        \item $A^N_\Pi \leq A$
        \item $A^\nabla_\nabla \leq A$
    \end{enumerate}
    \end{multicols}
\end{proposition}
\begin{proof}
    Again we prove these in the order presented. 
    \begin{enumerate}
        \item Follows immediately from the unit of the adjunction $\Id \leq (-)_\Delta (-)^\Delta$. 
        \item Same as (1), it follows from the adjunction unit of $(-)^\Pi \dashv (-)_N$. 
        \item This follows now immediately from the adjunction counit $(-)_\Pi (-)^N \leq \Id$. 
        \item As for (3), this follows from the adjunction counit of the adjunction $(-)_\nabla \dashv (-)^\nabla$. \qedhere
    \end{enumerate}
\end{proof}

\begin{remark}
    The fact that there is a sequence of natural transformations $(-)_\Pi (-)^N \leq \Id \leq (-)_N (-)^\Pi$ was used by Yao in \cite{yao_2004} to introduce rough set structure on formal contexts -- though not phrased in the categorical language.
\end{remark}

\section{Compositions and closures}

We have now constructed the eight possibilistic operators in formal concept analysis, as studied in \cite{dubois-dupin-prade_2007, dubois-prade_2009}, by standard categorical constructions associated to the underlying boolean profunctor $R\: G \nrightarrow G$. It remains to study the different possible compositions of upper-case and lower-case possibilistic operators. 

Given our eight operators, there are $16$ possible compositions that produce an endofunctor on $[G\op, \2]$. As we only consider endofunctors on $[G\op, \2]$ -- the ones on $[M,\2]\op$ are completely analogous -- we will use the combined notation $(-)_\sigma (-)^\rho := (-)^\sigma_\rho$ for simplicity. We then get the following table of compositions:

\begin{table}[ht]
\centering
\[
\begin{array}{c|cccc}
 & (-)_\Delta & (-)_\Pi & (-)_N & (-)_\nabla \\ \hline
(-)^\Delta
    & (-)^\Delta_\Delta
    & (-)^\Delta_\Pi
    & (-)^\Delta_N
    & (-)^\Delta_\nabla \\[0.4em]
(-)^\Pi
    & (-)^\Pi_\Delta
    & (-)^\Pi_\Pi
    & (-)^\Pi_N
    & (-)^\Pi_\nabla \\[0.4em]
(-)^N
    & (-)^N_\Delta
    & (-)^N_\Pi
    & (-)^N_N
    & (-)^N_\nabla \\[0.4em]
(-)^\nabla
    & (-)^\nabla_\Delta
    & (-)^\nabla_\Pi
    & (-)^\nabla_N
    & (-)^\nabla_\nabla
\end{array}
\]
\caption{The sixteen possible binary compositions.}
\label{tab:possibilistic-compositions}
\end{table}

The goal of this section is to prove that only three of the $16$ compositions that form closure operators on $P(G)$. First, we introduce some theory that allows us to prove this. 

\subsection{Idempotency and closure}

The standard FCA operator $(-)^\Delta_\Delta$ is a so-called closure operator, defined as follows. 

\begin{definition}
    A functor $F\: \C\to \C$ on a poset $\C$ is a \emph{closure operator} if it is extensive, monotone and idempotent, meaning it satisfies the following axioms:
\begin{enumerate}
    \item $A\leq F(A)$,
    \item $A\leq A' \implies F(A)\leq F(A')$, and
    \item $FF(A) = F(A)$.
\end{enumerate}
\end{definition}

The FCA adjunction coming from a formal context $(G,M,R)$ satisfies all three of these: by \cite[2.1.11]{jarvis_2025}, the monad $R_*R^*$ is idempotent, and the two other properties follow formally by the adjunction unit and the fact that it is a functor on posets. Thus, on a poset, any monadic functor determines a closure operator, and conversely, any closure operator is the monad of an adjunction -- see \cite[Chapter 7]{davey-priestley_2002} for details. 

Consequently, in order to be recognized as a closure operator on \emph{some} formal context, these three axioms have to be fulfilled. Otherwise the resulting fixedpoints have no chance of being formal concepts, or concepts of any type. In light of this discussion, the following proposition is rather trivial, but we state it for reference later.

\begin{proposition}
    \label{prop:axioms_for_closure}
    Let $L\: \C \to \C$ be an endofunctor on a poset. If $L$ fails to be extensive, monotone or idempotent, then it cannot be the closure operator of a formal context. 
\end{proposition}

This result will help us quickly identify which of the asymmetric compositions later possibly could give rise to alternative formal concepts on some related formal context. However, satisfying these three properties does not, as far as we know, guarantee that the fixedpoints form the formal concept of some formal context. 

\begin{remark}
    Dually, given an adjoint pair $L\dashv R$, the composition $LR$ is a comonad. On posets, this comonad satisfies dual properties of the closure operators, called being an \emph{interior} operator. In the case of formal concepts, instead of completing a subset $A$ to a formal concept, i.e., constructing the smallest formal concept containing the set, the interior operator finds the largest formal concept \emph{inside} $A$. 
\end{remark}

\subsection{Binary compositions}

We can now classify exactly which binary compositions of the eight operators that give closure operators. 

\begin{theorem}
    Out of the 16 operators in \cref{tab:possibilistic-compositions}, only $(-)^\Delta_\Delta$ and $(-)^\Pi_N$ are closure operators. 
\end{theorem}
\begin{proof}
    The monads of the adjunctions $(-)^\Delta \dashv (-)_\Delta$ and $(-)^\Pi \dashv (-)_N$, i.e., the operators $(-)^\Delta_\Delta$ and $(-)^\Pi_N$ are both closure operators by \cite[2.1.11]{jarvis_2025}. Similarly, due to the reversal of the adjunctions, $(-)^\Pi_N$ and $(-)^\nabla_\nabla$ are interior operators -- not closure operators. Hence, we only need to check the remaining $12$ cases. We check these by counterexamples, using \cref{prop:axioms_for_closure} to test whether $F$ satisfies the properties. In fact, it will be enough to test whether $F$ is extensive. 

    First, let $G=\{a,b\}$, $M=\{x,y\}$ and $R$ given by $aRx, bRy$. Chose the subset $A=G$. Then we have $A^\Delta = \varnothing$, $A^\Pi = \{x,y\}$, $A^N=  \{x,y\}$ and $A^\nabla = \varnothing$. In this example, both $A^\Delta_\Pi = \varnothing$ and $A^\Delta_N = \varnothing$, showing that neither operators $(-)^\Delta_\Pi$ or $(-)^\Delta_N$ are closure operators. Similarly, we get $A^\Pi_\Delta = \varnothing$, $A^\Pi_\nabla=\varnothing$, $A^N_\Delta = \varnothing$, $A^N_\nabla = \varnothing$, $A^\Delta_\Pi = \varnothing$ and $A^\nabla_N = \varnothing$, showing that none of these can be closure operators as well. 

    It remains to show the same for $(-)^\Pi_\Pi$, $(-)^\Delta_\nabla$, $(-)^N_N$ and $(-)^\nabla_\Delta$. For the former two we, use the formal context $G=\{a\}$, $M=\varnothing$ and $A=G$. This gives $A^\Pi_\Pi = \varnothing$ and $A^\Delta_\nabla = \varnothing$, showing neither is extensive and hence not a closure operator. Similarly, using $G=\{a\}$, $M=\{x\}$, $aRx$ and $A=G$ $A^\nabla_\Delta = \varnothing$. For the final operator, $(-)^N_N$ we use the context $G=\{a,b\}$, $M=\{x\}$, $aRx$, $bRx$ and $A=\{a\}$. This gives $A^N_N = \varnothing$, which also finishes the proof. 
\end{proof}

\subsection{Other closure operations}

We round off this section by using the categorical machinery to identify a sequence of new types of concepts, or pairs, that naturally arise from the framework. We do not intend to study these in detail here, but leave this for future work. Now, given two adjunctions 
\begin{center}
\begin{tikzcd}
    \C \arrow[r, yshift=2pt, "F_1"] & \D \arrow[r, yshift=2pt, "F_2"] \arrow[l, yshift=-2pt, "G_1"] & \E \arrow[l,yshift=-2pt, "G_2"]
\end{tikzcd}
\end{center}
the composed functor $F_2 F_1$ is a left adjoint to $G_2 G_1$. Recall from \cref{prop:axioms_for_closure} that on complete lattices, closure operators are equivalent to idempotent monads. This gives the following result. 

\begin{proposition}
    The functors $(-)^\Delta$ and $(-)^\Pi$ on $P(G)$, and  similarly $(-)_\Pi$ and $(-)_\nabla$ on $P(M)\op$ are left adjoints. Any composition of these will produce a closure operator on $P(G)$ by composing with the corresponding right adjoint. 
\end{proposition}

\begin{example}
    \label{ex:new-operator}
    The functor 
    \[(-)^\Delta (-)_\Pi (-)^\Delta \: [G\op, \2]\to [M,\2]\op\]
    is a composition of left adjoints, hence is a left adjoint itself, with right adjoint given by $(-)_\Delta (-)^N (-)_\Delta$. The composition of these is by \cite[2.1.11]{jarvis_2025} and standard categorical arguments a closure operator. 
\end{example}

\begin{proposition}
    If $C$ is the closure operator defined in \cref{ex:new-operator}, then there is no formal context $(G',M',R')$ such that $C = (-)_{\Delta_{R'}}(-)^{\Delta_{R'}}$. 
\end{proposition}
\begin{proof}
    Any FCA closure operator can be computed poitwise via $A^\Delta_\Delta = \bigcap_{g\in A}g^\Delta_\Delta$. The operator $C$ does not always satisfy this property, as we can see by the counterexample $G=\{a,b\}$, $M=\{x,y\}$ with $aRy$, $bRx$. Here $C(G)=G$, while $C(\{a\})\cap C(\{b\})=\varnothing$. 
\end{proof}

Another nice feature of the categorical language is that one can prove the existence of other adjoints, which are not one of the eight possibilistic operators. For example, as $(-)_\Delta$ is a right Kan extension, it preserves limits. As $[M,\2]\op$ is a complete lattice, the adjoint functor theorem applies, meaning it has a \emph{left} adjoint 
\[L^\Delta\: [G\op,\2]\to [M,\2]\op.\] 
This is not one of the possibilistic operators, and does not coincide with mixing quantifiers with the relation $R$ and its complement $\oR$. The composition $(-)_\Delta L^\Delta$ is a closure operator on $P(G)$ due to the general property of being an adjoint pair on posets. 

Similarly, the functor $(-)_\nabla$ is the dual of a left Kan extension, hence preserve limits, meaning that it has a left adjoint $L^\nabla$ by the adjoint functor theorem. The composition $(-)_\nabla L^\nabla$ is, by the adjunction property, another closure operator on $P(G)$.

\section{Conclusion}

This paper shows that the eight possibilistic operators studied in \cite{dubois-dupin-prade_2007,dubois-prade_2009} arise canonically as Kan extensions of profunctors. We used this to identify some of their properties as arising from standard categorical arguments, as well as classified exactly which binary compositions, asymmetric or symmetric, which give closure operators. Finally we introduced infinitely many new closure operators on $P(G)$ by using ideas from category theory. We are not aware of any studies on these operators. 

The categorical approach opens up a new avenue of research, and a new way of understanding existing results. In \cite{djouadi_prade_2011}, the authors extend the eight possibilistic operators in formal concept analysis to its fuzzy sibling. These results can be obtained in this categorical framework by replacing the boolean category $\2$ by a fuzzy lattice $L$. More generally, by enriching in an epistemic uncertainty framework, as introduced in our previous work \cite{aambo_2026}, one can obtain a plethora of extensions to this paper into epistemically uncertain settings. 

%Another direction worth further inquiry is the geometric picture. In \cite{jarvis_2025}, Jarvis introduces a monoidal structure on the subcategory of formal concepts. As formal concepts form a reflexive subcategory, one could attempt to study these reflexive ``ideals'' in $[G\op,\2]$. In particular, one could take inspiration from Balmers work on tensor-triangulated geometry \cite{balmer_2005} to define \emph{prime} ideals, and study the corresponding spectrum object. 

\printbibliography{}  

@article{eilenberg-maclane_1945,
	author = {Samuel Eilenberg and Saunders MacLane},
	journal = {Transactions of the American Mathematical Society},
	pages = {231--294},
	title = {General Theory of Natural Equivalences},
	year = {1945},
  volume = {58}
}

@book{maclane_1998,
  author = {Saunders MacLane},
  title = {Categories for the working mathematician},
  isbn = {0-387-98403-8},
  year = {1998},
  publisher = {Springer-Verlag, New York},
  series = {Graduate Texts in Mathematics},
  number = {5},
}

@book{kelly_1982,
  author = {Gregory Maxwell Kelly},
  title = {Basic concepts of enriched category theory},
  year = {1982},
  publisher = {Cambridge University Press},
  series = {London Mathematical Society Lecture Note Series},
  number = {64},
}

@misc{aambo_2026,
      title={A categorical formalization of epistemic uncertainty frameworks}, 
      author={Torgeir Aambø},
      year={2026},
      eprint={2603.04188},
      archivePrefix={arXiv},
      primaryClass={math.CT},
      url={https://arxiv.org/abs/2603.04188}, 
}

@book{dubois-prade_1988,
  author = {Didier Dubois and Henri Prade},
  title = {Possibility Theory. An Approach to Computerized Processing of Uncertainty},
  year = {1988},
  publisher = {Plenum Press, New York and London},
}

@incollection{dubois-prade_1998,
  author = {Didier Dubois and Henri Prade},
  booktitle = {Handbook of Defeasible Reasoning and Uncertainty Management Systems},
  title = {Possibility Theory: Qualitative and Quantitative Aspects},
  year = {1998},
  volume = {1},
  publisher = {Springer, Dordrecht},
}

@incollection{dubois-prade_2015,
  author = {Didier Dubois and Henri Prade},
  booktitle = {Springer Handbook of Computational Intelligence},
  title = {Possibility Theory and Its Applications: Where Do We Stand?},
  year = {2015},
  publisher = {Springer-Verlag, Berlin, Heidelberg},
  series = {Springer Handbooks},
}

@book{ganter-wille_1999,
  author = {Bernhard Ganter and Rudolph Wille},
  title = {Formal Concept Analysis: Mathematical Foundations},
  year = {1999},
  publisher = {Springer-Verlag},
}

@article{dubois-dupin-prade_2007,
  author = {Didier Dubois and Florence Dupin de Saint-Cyr and Henri Prade},
  journal = {Fundamenta Informaticae},
  title = {A possibility-theoretic view of formal concept analysis},
  year = {2007},
  volume = {75},
  number = {1},
  pages = {195--213},
}

@inproceedings{dubois-prade_2009,
	location = {Lisbonne, Portugal},
	title = {Possibility theory and formal concept analysis in information systems},
	pages = {1021--1026},
	booktitle = {Proceedings of the Joint 2009 International Fuzzy Systems Association World Congress and 2009 European Society of Fuzzy Logic and Technology Conference},
  ISBN = {978-989-95079-6-8},
	publisher = {European Society for Fuzzy Logic and Technology ({EUSFLAT})},
	author = {Dubois, Didier and Prade, Henri},
	year = {2009},
}

@inproceedings{duntsch-gediga_2002,
	title = {Modal-style operators in qualitative data analysis},
	doi = {10.1109/ICDM.2002.1183898},
	eventtitle = {2002 {IEEE} International Conference on Data Mining},
	pages = {155--162},
	booktitle = {Proceedings of the 2002 {IEEE} International Conference on Data Mining},
	author = {Duntsch, N. and Gediga, G.},
	year = {2002},
}

@inproceedings{yao_2004,
	location = {Berlin, Heidelberg},
	title = {A Comparative Study of Formal Concept Analysis and Rough Set Theory in Data Analysis},
	isbn = {978-3-540-25929-9},
	doi = {10.1007/978-3-540-25929-9_6},
	pages = {59--68},
	booktitle = {Rough Sets and Current Trends in Computing},
	publisher = {Springer},
	author = {Yao, Yiyu},
	year = {2004},
}

@inproceedings{yao-chen_2006,
	location = {Berlin, Heidelberg},
	title = {Rough Set Approximations in Formal Concept Analysis},
	isbn = {978-3-540-39383-2},
	doi = {10.1007/11847465_14},
	pages = {285--305},
	booktitle = {Transactions on Rough Sets V},
	publisher = {Springer},
	author = {Yao, Yiyu and Chen, Yaohua},
	year = {2006},
}

@inproceedings{yakub-djouadi-dubois-prade_2016,
	location = {La Hague, Netherlands},
	title = {From a Possibility Theory View of Formal Concept Analysis to the Possibilistic Handling of Incomplete and Uncertain Contexts},
	pages = {pp. 79--88},
	number = {1703},
	booktitle = {Proceedings of the 5th International Workshop ``What can {FCA} do for Artificial Intelligence?'' ({FCA4AI}), co-located with the European Conference on Artificial Intelligence {ECAI}},
	author = {Ait-Yakoub, Zina and Djouadi, Yassine and Dubois, Didier and Prade, Henri},
  year = {2016},
}

@article{djouadi_prade_2011,
	title = {Possibility-theoretic extension of derivation operators in formal concept analysis over fuzzy lattices},
	volume = {10},
	doi = {10.1007/s10700-011-9106-5},
	pages = {287--309},
	number = {4},
	journaltitle = {Fuzzy Optimization and Decision Making},
	author = {Djouadi, Yassine and Prade, Henri},
	year = {2011},
}

@phdthesis{jarvis_2025,
	title = {A Novel Closed Monoidal Structure on the Nucleus of a Profunctor},
	url = {https://academicworks.cuny.edu/gc_etds/6231},
  type = {PhD thesis},
	author = {Jarvis, Samantha},
  year = {2025},
  institution = {The Graduate Center, City University of New York,}
}

@phdthesis{ferrer_2023,  
  author = {Lance Ferrer},
  title = {Concept Analysis in Categories},
  type = {PhD thesis},
  institution = {University of Hawaii at Manoa},
  url = {https://hdl.handle.net/10125/107917},
  year = {2023},
  }

@book{davey-priestley_2002,
  author    = {B. A. Davey and H. A. Priestley},
  title     = {Introduction to Lattices and Order},
  edition   = {2},
  publisher = {Cambridge University Press},
  year      = {2002}
}

@article{singh-kumar-gani_2016,
	title = {A comprehensive survey on formal concept analysis, its research trends and applications},
	volume = {26},
	doi = {10.1515/amcs-2016-0035},
	pages = {495--516},
	number = {2},
	journaltitle = {International Journal of Applied Mathematics and Computer Science},
	publisher = {University of Zielona Góra},
	author = {Singh, Prem Kumar and Kumar, Cherukuri Aswani and Gani, Abdullah},
	year = {2016},
}
\end{document}